\documentclass[reqno]{amsart}

\usepackage{amsmath,amssymb,amscd,accents} \usepackage{graphicx}
\DeclareGraphicsExtensions{.eps} \usepackage{mathrsfs}
\usepackage[mathcal]{eucal} \usepackage{esint} \usepackage{color}

\newtheorem{Thm}{Theorem}{\bfseries}{\itshape}
\newtheorem*{Thm*}{Theorem}{\bfseries}{\itshape}
\newtheorem{Cor}{Corollary}{\bfseries}{\itshape}
\newtheorem{Prop}[Cor]{Proposition}{\bfseries}{\itshape}
\newtheorem{Lem}[Cor]{Lemma}{\bfseries}{\itshape}
\newtheorem*{Lem*}{Lemma}{\bfseries}{\itshape}
{\bfseries}{\itshape}
{\bfseries}{\itshape}
\newtheorem{Def}[Cor]{Definition}{\bfseries}{\rmfamily}
{\scshape}{\rmfamily}
\newtheorem{Rem}[Cor]{Remark}{\scshape}{\rmfamily}
{\scshape}{\rmfamily}
{\bfseries}{\itshape}

\renewcommand\ge{\geqslant} \renewcommand\le{\leqslant}
\let\tildeaccent=\~ \let\hataccent=\^
\renewcommand\~[1]{\widetilde{#1}}

\def\<{\left<} \def\>{\right>} \def\({\left(} \def\){\right)}

 \def\norm#1{\left\Vert #1
  \right\Vert} 

\let\parasymbol=\S \def\secref#1{\parasymbol\ref{#1}}
 \def\pd#1#2{\tfrac{\partial#1}{\partial#2}}

\let\polishL=l \def\Zoladek.{\.Zol\c adek}

 \def\Im{\operatorname{Im}}

\def\etc.{\emph{etc}.}

\def\:{\colon} \def\R{{\mathbb R}} \def\C{{\mathbb C}} \def\Z{{\mathbb
    Z}} \def\N{{\mathbb N}} \def\Q{{\mathbb Q}} \def\P{{\mathbb P}}

\def\A{{\mathbb A}} 

\let\PolishL=\L 
\def\L{{\mathbb L}}

 \def\e{\varepsilon} \def\S{\varSigma}

\def\poly{\operatorname{poly}}

 \def\Lojas.{\PolishL ojasiewicz}      
  
  \def\cL{{\mathcal L}}

  \def\cC{{\mathcal C}}

\def\rest#1{{\vert_{#1}}}

\begin{document}

\title[Bounds for rational points on algebraic curves]{Bounds for rational points on algebraic curves, optimal in the degree, and dimension growth}

\author[Binyamini]{Gal Binyamini} \address{Weizmann Institute of
  Science, Rehovot, Israel} \email{gal.binyamini@weizmann.ac.il}

\author[Cluckers]{Raf Cluckers} \address{Univ.~Lille, CNRS, UMR 8524 -
  Laboratoire Paul Painlev\'e, F-59000 Lille, France, and KU Leuven,
  Department of Mathematics, B-3001 Leu\-ven, Bel\-gium}
\email{Raf.Cluckers@univ-lille.fr}
\urladdr{http://rcluckers.perso.math.cnrs.fr/}

\author[Novikov]{Dmitry Novikov} \address{Weizmann Institute of
  Science, Rehovot, Israel} \email{dmitry.novikov@weizmann.ac.il}

\thanks{G.B. was supported
  by funding from the European Research Council (ERC) under the
  European Union's Horizon 2020 research and innovation programme
  (grant agreement No 802107). R.C.~was partially supported by KU
  Leuven IF C16/23/010 and the Labex CEMPI (ANR-11-LABX-0007-01). D.N
  was supported by the ISRAEL SCIENCE FOUNDATION (grant No. 1167/17)
  and by funding received from the MINERVA Stiftung with the funds
  from the BMBF of the Federal Republic of Germany. }


\begin{abstract} 
  Bounding the number of rational points of height at most $H$ on irreducible
  algebraic plane curves of degree $d$ has been an intense topic of
  investigation since the work by Bombieri and
  Pila. In this paper we establish optimal dependence on $d$, by showing the upper bound
  $C d^2 H^{2/d} (\log H)^\kappa$ with some absolute constants $C$ and $\kappa$.
This bound is optimal with respect to both $d$ and $H$,
  except for the constants $C$ and $\kappa$. This answers a
  question raised by Salberger, leading to a simplified proof
  of his results on the uniform dimension growth conjectures of
  Heath-Brown and Serre, and where at the same time we replace the
  $H^\e$ factor by a power of $\log H$.

  The main strength of our approach comes from the combination of a new, efficient
  form of smooth parametrizations of algebraic curves with a century-old
  criterion of P\'olya, which allows us to save one extra power of $d$ compared with the
  standard approach using B\'ezout's theorem.
\end{abstract}

\maketitle

\section{Introduction}

Upper bounds for the number of rational points of bounded height on
algebraic curves of degree $d$ have been studied intensely and used broadly ever since the
renowned results by Bombieri and Pila \cite{bombieri-pila}. In particular,
recently, a polynomial dependence on $d$ in the bounds is obtained in \cite{CCDN-dgc}. In this paper
we sharpen this polynomial dependence on $d$ to a quadratic dependence (see Theorem \ref{thm:main}), answering a question by Salberger from \cite{Salberger-dgc}. Such a quadratic dependence on $d$
is known to be optimal, by \cite{CCDN-dgc}. Let us make all this precise and sketch our proof ingredients, starting with Salberger's question.

While proving the dimension growth conjecture, Salberger \cite[
below Theorem 0.12]{Salberger-dgc} raises the question whether
there is $c=c(\varepsilon,n)$ such that for any algebraic integral
curve $C$ in $\P^n_\Q$ of degree $d$ and any $H\ge 1$, $\varepsilon>0$ one has
\begin{equation}\label{eq:Sal}
  \# C(\Q,H) \le cd^{2+\varepsilon}H^{2/d+\varepsilon},
\end{equation}
where $C(\Q,H)$ denotes the set of rational points on $C$ of height at
most $H$. 
He furthermore suggests that a positive answer would simplify vastly
his proof on uniform dimension growth
\cite[Theorems~0.3, 0.4]{Salberger-dgc}, which goes back to questions raised
by Heath-Brown \cite[p.~227]{Heath-Brown-cubic} and Serre
\cite[page 178]{Serre-Mordell}, \cite[p.~27]{Serre-Galois}, about
bounds for the number of rational points of bounded height lying on an
irreducible variety, in terms of its dimension.

The closest result to Salberger's suggested estimates from (\ref{eq:Sal})
prior to this work is the upper bound from \cite[Theorem 2]{CCDN-dgc}
which gives
\begin{equation}\label{eq:d4}
  \# C(\Q,H) \le cd^{4}H^{2/d},
\end{equation}
for some $c=c(n)$ depending only on
$n$. 
Furthermore, the optimality of Salberger's question about (\ref{eq:Sal}), apart
from $\varepsilon$, is shown in \cite[Proposition 5]{CCDN-dgc} where an
integral curve $C$ in $\P^2_\Q$ of any given degree $d$ is constructed for which the
lower bound
\begin{equation}\label{eq:lower}
  c' d^{2}H^{2/d}\le \# C(\Q,H)
\end{equation}
holds for some $H$ close to $d$ and
with a universal constant $c'$.

In
this paper we show that Salberger's question has a positive answer, with
$\varepsilon=0$ and just a polynomial factor in $\log H$, see Theorems \ref{thm:main} and \ref{thm:main:gen} below. Furthermore, as an application we explain how
this gives a
short, streamlined proof of Salberger's result on
uniform dimension growth \cite[Theorem~0.4]{Salberger-dgc} while furthermore replacing $H^\varepsilon$ by a power of $\log H$,
see Theorems \ref{thm:dcgdegree} and
\ref{thm:0.4} in Section \ref{sec:dgc} below. 




\subsection{}
Let us now introduce our main result, Theorem \ref{thm:main}, and sketch our proof techniques with new, efficient
  forms of smooth parametrizations of algebraic curves (see Section \ref{sec:param}) and an application of P\'olya's criterion (see Sections \ref{sec:chebyshev} and \ref{sec:chebyshev-Gamma}).

Let $P\in\R[x,y]$ be irreducible of degree $d>0$ and
set $\Gamma\subset\R^2$ to be $\Gamma:=\{P=0\}$.  Write $\Gamma(\Q,H)$
for the set of rational points on $\Gamma$ with coordinates of height
at most $H$ (that is, each coordinate is a quotient $a/b$ of integers
with $|a|\le H$ and $|b|\le H$). We will prove the following
theorem. 


\begin{Thm}\label{thm:main}
  For every $H\ge 2$, one has
  \begin{equation}\label{eq:main}
    \#\Gamma(\Q,H) \le c d^2 H^{2/d} (\log H)^\kappa
  \end{equation}
  where $c$ and $\kappa$ are universal
  constants.  
\end{Thm}

The bounds (\ref{eq:Sal}) follow from Theorem \ref{thm:main} (see Section \ref{sec:proof1}), and also the bounds (\ref{eq:main}) are optimal except for the constants
$c$ and $\kappa$, by \cite[Section 6]{CCDN-dgc}.

We briefly sketch the main argument for Theorem \ref{thm:main}, which is worked out in Sections \ref{sec:chebyshev} up to \ref{sec:proof1}.
Set $k:=\log H$ and $\Gamma':=\Gamma\cap[0,1]^2$. The case $k\ge d$
follows from the bounds (\ref{eq:d4}), so assume $k<d$. We
show in Section \ref{sec:param}, by establising an efficient smooth parametrization result for algebraic curves which refines
\cite{wilkie:pw-notes,me:pfaff-wilkie}, that $\Gamma'$ can be
parameterized using $d^2\poly(k)$ maps $\{\phi_j:(0,1)\to\Gamma'\}$
with unit $C^r$-norms where $r\sim k^2$ and where $\poly(k)$ means polynomial dependence in $k$, more precisely of the form $ck^e$ for some $c$ and $e$ (independent of $d$ and $\Gamma$). It then follows, essentially
by the Bombieri-Pila method, that $\Gamma(\Q,H)$ is contained in
$d^2 \poly(k)$ curves $\{Q_j=0\}$ of degree $k$. Note that the
$H^{2/d}$ factor that normally appears here is $O(1)$ when $k<d$. 

Since $k<d$ the intersection $\Gamma\cap\{Q_j=0\}$ is finite.  In the
standard approach, one would proceed from here by bounding the number of points on each $\Gamma\cap\{Q_j=0\}$ by $dk$ from
B\'ezout's theorem, giving the bound
\begin{equation}
  \#\Gamma(\Q,H) \le d^3 \poly(k).
\end{equation}
A key strength of our approach is an improvement of this final step,
replacing $dk$ from B\'ezout essentially by $k^2$.

To achieve this, we consider the restriction of the space of
polynomials of degree $k$ in $\R[x,y]$ to $\Gamma$. This is a linear
space, and a classic theorem of P\'olya \cite{polya} gives a criterion
for such a space to form a \emph{Chebyshev system},
see~\secref{sec:chebyshev}.

P\'olya shows that for $\cL$ to form a Chebyshev system it is enough
to check that certain Wronskians are non-vanishing. In our context,
these Wronskians turn out to be given by polynomials of degree
$d\cdot\poly(k)$ in $\R[x,y]$, none of them identically vanishing on
$\Gamma$. Removing all their zeros from $\Gamma$ and treating the
complement, we divide $\Gamma$ into $d^2\poly(k)$ pieces $\Gamma_j$
such that
\begin{equation}
    \# \big(   \Gamma_j\cap\{Q=0\} \big) < \frac{(k+1)(k+2)}2 \qquad \text{whenever
  } \deg Q\le  k.
\end{equation}
This eliminates one extra power of $d$ from the bound and allows us to finish the
proof of Theorem \ref{thm:main} in Section \ref{sec:proof1}.

\subsection{}Let us mention some extra context.
Natural variants for counting inside other fields than $\Q$ are
recently obtained in \cite{Sedunova,Vermeulen:p,Liu:Chunhui,Pared-Sas} for global fields, and, motivically in $\C((t))$, in \cite[Section 5]{CCL-PW}. Unfortunately, our methods do not seem to adapt to these more general settings as we use real analytic functions for both steps in our proof of Theorem \ref{thm:main}, namely the efficient smooth parameterizations, and, P\'olya's criterion.

Bounds for rational points of bounded height on curves lead to effective forms of Hilbert's irreducibility theorem, see
\cite{Walkowiak,DebesW,Pared-Sas:Hilbert}; our Theorem \ref{thm:main} with its quadratic dependence on $d$ leads naturally to corresponding improvements.

Question (\ref{eq:Sal}) about counting rational points of bounded height on algebraic curves and our positive answer to this question
sits in a rich history of results and techniques from
e.g.~\cite{bombieri-pila,Pila-density-curve,Heath-Brown-Ann,Broberg,Walkowiak,Walsh,CCDN-dgc,me:c-cells}
towards better control in $d$ and removal of $\varepsilon$. P\'olya's criterion and the ideas behind it have been used already many times in this context (see e.g.~\cite{Schmidt-points}, \cite{pila:density-Q}), but, it seems to be the first time P\'olya's criterion  is used to achieve optimal dependence on the degree.

Note that our Theorems \ref{thm:dcgdegree} and \ref{thm:0.4} settle Serre's question on \cite[page 178]{Serre-Mordell}
with $K=\Q$ for all degrees $d\ge 4$, the degree $4$ case being new, and sharpens our understanding of the degree $3$ case.

\subsection{Acknowledgments}

The authors G.B. and R.C. would like to thank the Royal Swedish
Academy for their hospitality during the Schock Prize Symposium 2022 for Jonathan Pila, where
this work began. It is also our pleasure to thank Per Salberger for
fruitful discussions during this event, and in particular for pointing
out the importance of the estimate~\eqref{eq:Sal} which is the main
topic of this paper. The author R.C. would like to thank Tim Browning, Wouter
Castryck, Philip Dittmann, Marcelo Paredes, Jonathan Pila, Per Salberger, and Roman
Sasyk for interesting discussions on the topics of the paper.


\section{Chebyshev systems and P\'olya's criterion}
\label{sec:chebyshev}

Let $f_1,\ldots,f_\mu:(a,b)\to\R$ be real analytic functions.
\begin{Def}[Partial Wronskians]\label{def:partial-wronskian}
  We define the partial Wronskian $W_j:(a,b)\to\R$ to be
  \begin{equation}
    W_j = W(f_1,\ldots,f_j) = \det
    \begin{pmatrix}
      f_1 & \cdots & f_j \\
      f'_1 & \cdots & f'_j \\
      \vdots & \ddots & \vdots \\
      f_1^{(j-1)} & \cdots & f_j^{(j-1)}
    \end{pmatrix}.
  \end{equation}
\end{Def}

Recall the notion of a Chebyshev system.

\begin{Def}
  An $\R$-linear space $\cL$ of real-valued function functions on an
  interval $(a,b)$ is said to be a \emph{Chebyshev system} if the
  number of isolated zeros of any function in $\cL$ does not exceed
  $\dim\cL-1$.
\end{Def}

The following criterion for $\cL$ to be a Chebyshev system plays a
major role in our argument.

\begin{Thm*}[P\'olya's theorem \protect{\cite{polya}}]
  Let $f_1,\ldots,f_\mu\in\cL$ be an $\R$-basis for $\cL$, and suppose
  that the corresponding partial Wronskians $W_j$ are nowhere
  vanishing on $(a,b)$ for $j=1,\ldots,\mu$. Then $\cL$ is a Chebyshev
  system.
\end{Thm*}

We note that \cite{polya} does not give this exact formulation, though
it easily follows from the material presented there. For a full proof
see \cite[P\'olya's theorem,~p.913]{ny:rolle}.

\section{Chebyshev systems on $\Gamma$}
\label{sec:chebyshev-Gamma}

Let $P\in\R[x,y]$ be irreducible of degree $d$ and set
$\Gamma\subset\R^2$ to be $\Gamma:=\{P=0\}$. We denote by
$L_P:\R[x,y]\to\R[x,y]$ the differential operator
\begin{equation}
  L=L_P = P_y \pd{}x - P_x \pd{}y.
\end{equation}
We may also think of $L$ as a vector field $\xi:\R^2\to T\R^2$. Note
that $\xi$ restricts to a vector field $\xi:\Gamma\to T\Gamma$.

Fix some $k<d$. Denote by $\R[x,y]_{\le k}$ the set of
polynomials of degree at most $k$ and set
\begin{equation}\label{eq:mu}
  \mu=\mu(k)=\dim\R[x,y]_{\le k} = \frac{(k+2)(k+1)}2.
\end{equation}

Let $f_1,\ldots,f_\mu$ denote a basis, for example the monomial basis,
for $\R[x,y]_{\le k}$.

\begin{Def}[Partial $\xi$-Wronskians]\label{def:partial-xi-wronskian}
  We define the partial $\xi$-Wronskian $W_j\in\R[x,y]$ to be
  \begin{equation}
    W_j = W(f_1,\ldots,f_j) = \det
    \begin{pmatrix}
      f_1 & \cdots & f_j \\
      L f_1 & \cdots & L f_j \\
      \vdots & \ddots & \vdots \\
      L^{j-1} f_1 & \cdots & L^{j-1} f_j
    \end{pmatrix}.
  \end{equation}
\end{Def}

\begin{Rem}\label{rem:wronskians}
  Let $\phi:(a,b)\to\R^2$ be a trajectory of $\xi$, i.e. a curve
  satisfying
  \begin{equation}
    \phi(t)'=\xi(\phi(t)).
  \end{equation}
  Then for $t\in(a,b)$, the partial $\xi$-Wronskian $W_j$ at $\phi(t)$
  agrees by definition with the partial Wronskian of
  $f_1\circ\phi,\ldots,f_j\circ\phi$ at $t$.
\end{Rem}

We have the following elementary degree estimate.

\begin{Prop}\label{prop:Wj-degs}
  For $j=1,\ldots,\mu$,
  \begin{equation}
    \deg W_j \le j(k+jd).
  \end{equation}
\end{Prop}
\begin{proof}
  By direct computation.
\end{proof}

Let $\Sigma_k\subset\R^2$ be given by
\begin{equation}
  \Sigma_k := \{P_x=P_y=P=0\} \cup \{W_1=0\}\cup\cdots\cup\{W_\mu=0\}.
\end{equation}

\begin{Lem}\label{lem:Sigma-size}
  We have
  \begin{equation}
    \#(\Gamma\cap\Sigma_k) \le d^2\poly(k).
  \end{equation}
\end{Lem}
\begin{proof}
  This follows from Proposition~\ref{prop:Wj-degs} and B\'ezout's
  theorem. Note that no $W_j$ can vanish identically on $\Gamma$, as
  by a standard result on Wronskians this would imply that the
  restrictions of $f_1,\ldots,f_j$ to $\Gamma$ are $\R$-linearly
  dependent. But such a linear combination would be a polynomial of
  degree $k$ vanishing identically on $\Gamma$, contradicting our
  assumption $k<d$.
\end{proof}

The following is our main proposition for this section.

\begin{Prop}\label{prop:comp-chebyshev}
  Let $\Gamma^\circ$ 
  be a connected
  component of $\Gamma\setminus\Sigma_k$. Then for any nonzero $Q\in\R[x,y]_{\le k}$ we have
  \begin{equation}
    \# \big( \Gamma^\circ\cap\{Q=0\} \big) < \mu.
  \end{equation}
\end{Prop}
\begin{proof}
  Let $f:(a,b)\to\Gamma^\circ$ be the time parametrization of $\Gamma$
  with respect to the vector field $\xi$, i.e.
  \begin{equation}
    f'(t) = \xi(f(t))
  \end{equation}
  where
  $a,b\in\R\cup\{-\infty,\infty\}$.  Indeed, a trajectory of
  $\xi$ passing through a point of
  $\Gamma^\circ$ is analytic and parameterizes all of
  $\Gamma^0$ by the theorem on existence and uniqueness of solutions of ordinary differential equation.  By Remark~\ref{rem:wronskians},
  $W_j\rest{\Gamma^\circ}$ is the classical Wronskian of the functions $f_1\circ
  f,\ldots,f_j\circ
  f$. By assumption, none of these Wronskians vanish on
  $(a,b)$. By P\'olya's theorem the functions $f_1\circ
  f,\cdots,f_\mu\circ
  f$ form a Chebyshev system on
  $(a,b)$. That is, every nontrivial linear combination, and in particular $Q\circ
  f$, has fewer than $\mu$ zeros.
\end{proof}

\section{Smooth parametrization}
\label{sec:param}

We continue with the notation of~\secref{sec:chebyshev-Gamma}. It will
be convenient in this section to use rational rescalings of $L,\xi$,
namely
\begin{align}
  L_x &:= \pd{}x - \frac{P_x}{P_y} \pd{}y, \\
  L_y &:= \pd{}y - \frac{P_y}{P_x} \pd{}x,
\end{align}
and the corresponding vector fields $\xi_x$, $\xi_y$.

\begin{Prop}\label{prop:Qxj-degs}
  Let $Q\in\R[x,y]$. Then for every $j\in\N$ there exist
  $Q_{xj},Q_{yj}\in\R[x,y]$ such that
  \begin{align}
    L_x^jQ = \frac{Q_{xj}}{P_y^{2j-1}}, \qquad \deg Q_{xj}\le \deg Q+2dj, \\
    L_y^jQ = \frac{Q_{yj}}{P_x^{2j-1}}, \qquad \deg Q_{yj}\le \deg Q+2dj.
  \end{align}
\end{Prop}
\begin{proof}
  Follows by induction on $j$ using
  \begin{multline}
    L_x^{j+1}y=L_x\Big(\frac{Q_{xj}}{P_y^{2j-1}}\Big)=
    \frac{P_yL_xQ_{xj}-(2j-1)Q_{xj}L_xP_y}{P_y^{2j}}=\\
    \frac{P_y^2(Q_{xj})_x+P_yP_x(Q_{xj})_y-(2j-1)Q_{xj}(P_yP_{yx}-P_xP_{yy})}{P_y^{2j+1}}
  \end{multline}
  and symmetrically for $L^j_y$.
\end{proof}

Let $r\in\N$ and let $\Pi_r\subset\Gamma$ be the collection of
isolated zeros of each of the following polynomials restricted to
$\Gamma$:
\begin{gather}
  P_x,P_y,P_x\pm P_y,x\pm1,y\pm1,\\
  y_{xj},x_{yj} \qquad \text{for }j=0,\ldots,r,
\end{gather}
with notation from Proposition \ref{prop:Qxj-degs} with $Q=y$ and $Q=x$.

The following is an immediate consequence of
Proposition~\ref{prop:Qxj-degs} and B\'ezout's theorem.

\begin{Lem}\label{lem:Pi-size}
  We have
  \begin{equation}
    \# \Pi_r \le d^2\poly(r).
  \end{equation}
\end{Lem}

Now follows the main proposition for this section.
\begin{Prop}\label{prop:comp-param}
  Let $\Gamma^\circ\subset(\Gamma\cap[0,1]^2)\setminus\Pi_r$ be a
  connected component. Then for $N=\poly(r)$ there exist $C^r$ maps
  $\phi_1,\ldots,\phi_N:(0,1)\to\Gamma^\circ$ with $C^r$-norm
  $\norm{\phi_j}_r\le1$ for $j=1,\ldots,N$ and
  \begin{equation}
    \Gamma^\circ = \bigcup_{j=1}^N \Im\phi_j.
  \end{equation}
\end{Prop}
\begin{proof}
  By definition of $\Pi_r$, we can assume that $\Gamma^\circ$ is the
  graph of a real-analytic function $f:J\to(0,1)$ with $J\subset(0,1)$
  and $|f'|<1$ uniformly over $J$; or similarly with the roles of $x$
  and $y$ reversed. Moreover $f,f',\ldots,f^{(r+1)}$ have constant
  sign in $J$ (possibly zero, if one of the $y_{xj}$ or $x_{yj}$
  vanish identically on $\Gamma$). The proposition now follows from this in a classic way, see e.g.~the
  proof of \cite[Theorem~5.7]{wilkie:pw-notes}, which yields no more than $N=\poly(r)$ many $C^r$ maps as desired. 
\end{proof}

\begin{Cor}\label{cor:interpolation}
  Let $\Gamma^\circ$ be as in Proposition~\ref{prop:comp-param} and
  $H\ge 2$, $k>1$. Then $\Gamma^\circ(\Q,H)$ is contained in a union of
  $\poly(k)H^{O(1/k)}$ algebraic curves of degree at most $k$ in
  $\R^2$.
\end{Cor}

\begin{proof}
  Let $r=\mu$ in the notations of~\secref{sec:chebyshev-Gamma}.
  Following the parametrization afforded by
  Proposition~\ref{prop:comp-param} the claim is now a standard fact,
  essentially due to Bombieri-Pila \cite{bombieri-pila}.
\end{proof}

Note that we will use Corollary \ref{cor:interpolation} only when $k=\log H$ and $k<d$, so that the factor $H^{O(1/k)}$ becomes irrelevant. 

\section{Proof of Theorem~\ref{thm:main}}\label{sec:proof1}

In this section we prove a natural generalization of
Theorem~\ref{thm:main} given as Theorem \ref{thm:main:gen}.  For an
affine variety $X$ write $X(\Z,H)$ for the set of points on $X$ with
integer coordinates of height (that is, absolute value) at most $H$,
and, $X(\Q,H)$ for the set of rational points on $X$ with coordinates
of height at most $H$.

\begin{Thm}\label{thm:main:gen}
  Let $C$ be an integral algebraic curve in $\A_K^n$ of degree $d>0$
  for some subfield $K$ of $\C$ and some $n>1$. There is a constant
  $c=c(n)$ depending only on $n$ and an absolute constant $\kappa$
  such that for every $H\ge 2$,
  \begin{equation}\label{eq:main:gen}
    \#C(\Q,H) \le c d^2 H^{2/d} (\log H)^\kappa,
  \end{equation}
  and,
  \begin{equation}\label{eq:aff:main}
    \# C(\Z,H) \le c d^{2}H^{1/d}(\log H)^\kappa.
  \end{equation}
\end{Thm}

\begin{proof}[Proof of Theorem \ref{thm:main:gen}]
  Set $k=\log H$. 
  If $k\ge d$ then Theorem~\ref{thm:main:gen} follows easily from
  known results, as follows.  If $K=\Q$, then the bounds (\ref{eq:d4})
  for the curve in $\P^n_\Q$ given as the projective closure of $C$
  imply
  \begin{equation}
    \#C(\Q,H) \le c d^4 H^{2/d}
  \end{equation}
  for some $c=c(n)$. Still with $K=\Q$, Theorem 4 of \cite{CCDN-dgc}
  gives
  \begin{equation}
    \#C(\Z,H) \le c d^4 (H^{1/d} + \log H/d).
  \end{equation}
  And thus, one can simply take $\kappa=2$ in the case $k\ge d$.  The
  case that $C$ is not defined over $\Q$ follows by taking an
  intersection of $C$ with a Galois conjugate of $C$, yielding
  $\#C(\Q,H)\le d^2$, by B\'ezout (see e.g.~\cite[proof of Lemma
  4.1.3]{CCDN-dgc} for details).
  We may thus from now on assume that $k < d$. Similar as in the prior
  discussion in this proof, it is enough to treat the case $n=2$ and
  $K=\R$, that is, the situation of Theorem \ref{thm:main}. Indeed, as
  in \cite[Lemmas 4.1.3, 4.1.4]{CCDN-dgc} one reduces to the case that
  $C$ is absolutely irreducible (and furthermore defined over $\Q$),
  and, by the case $m=1$ of \cite[Proposition 4.3.2]{CCDN-dgc} one
  reduces to the case $n=2$ by projecting. This finishes the proof of
  Theorem \ref{thm:main:gen} (and thus also Theorem \ref{thm:main}) up
  to proving the special case of Theorem \ref{thm:main} with $k<d$.
\end{proof}

We will now treat the key case and remaining part of Theorem
\ref{thm:main}.

\begin{proof}[Proof of Theorem \ref{thm:main} when $\log H < d$]
  We write $k=\log H$ and set $r=\mu\sim k^2$ with $\mu=\mu(k)$ as in (\ref{eq:mu}).  Let
  \begin{equation} 
    \Gamma' = (\Gamma\cap[0,1]^2)\setminus(\Sigma_k\cup\Pi_r).
  \end{equation}
  By Harnack's theorem, the curve $\Gamma$ has at most $d^2$
  connected components. By Lemmata~\ref{lem:Sigma-size}
  and~\ref{lem:Pi-size}, it follows that $\Gamma'$ has at most
  $d^2\poly(k)$ connected components. Let $\Gamma^\circ$ be one of
  them.

  According to Corollary~\ref{cor:interpolation},
  \begin{equation}
    \Gamma^\circ(\Q,H)\subset \cup_{j=1}^N \{Q_j=0\}
  \end{equation}
  where
  \begin{align}
    \deg Q_j &\le k, &   N=\poly(k)H^{2/k}=\poly(k).
  \end{align}
  According to Proposition~\ref{prop:comp-chebyshev},
  \begin{equation}
    \#\big(\Gamma^\circ\cap\{Q_j=0\}\big) \le \mu \le  \poly(k), \qquad j=1,\ldots,N
  \end{equation}
  and we conclude that $\#\Gamma^\circ(\Q,H)\le\poly(k)$. Note that
  here we crucially use the assumption $k<d$.

  Repeating this for every connected component, and adding the
  cardinality of $\Gamma\cap(\Sigma_k\cup\Pi_r)$ to be on the safe
  side, we find that
  \begin{equation}
    \#(\Gamma\cap[0,1]^2)(\Q,H) \le d^2\poly(k).
  \end{equation}
  By a standard argument using symmetries $x\to-x$ and $x\to1/x$ (and
  similarly with $y$) we cover the rest of $\R^2$, and the statement
  follows.
\end{proof}
Theorems \ref{thm:main} and \ref{thm:main:gen} are now fully
proved. Note that Theorem \ref{thm:main:gen} clearly implies the
bounds (\ref{eq:Sal}) 
from Salberger's question.


\section{Uniform dimension growth}\label{sec:dgc}

As an application, let us explain how Theorem \ref{thm:main} together with Proposition
\ref{prop4.2.1} (quoted below, from \cite{CCDN-dgc}) yields a short
proof of the uniform dimension growth result of Salberger's
\cite{Salberger-dgc} for $d\ge 3$ (see also \cite{Brow-Heath-Salb} for degree $d\ge 6$), while removing $H^\varepsilon$, which is new in
degrees $3$ and $4$. In degree $d\ge 5$, the $\varepsilon$ was already
removed in \cite[Theorems 1 and 4]{CCDN-dgc}.  
For a variety $X\subset \P_\Q^n$ let
$$
X(\Q, H)
$$
be the set of rational points of height at most $H$ on $X$. Recall
that the height of a $\Q$-rational point $x$ in $\P^n$ is given by
$$
\max(|x_0|, \ldots,|x_n|)
$$
for an $n+1$-tuple $(x_0 ,\ldots , x_n )$ of integers $x_i$ which form
homogeneous coordinates for $x$ and have greatest common divisor equal
to $1$.

\begin{Thm}[Uniform dimension growth, projective case]\label{thm:dcgdegree}
  Given $n>2$, there exist constants $c=c(n)$, $e=e(n)$ and an
  absolute constant $\kappa$, such that for all integral projective
  varieties $X\subset \P_\Q^n$ of degree $d \geq 3$ and all $H\geq 2$
  one has
  \begin{equation}\label{eq:dgc:dim X5}
    \#X(\Q,H) \leq c d^e H^{\dim X} \mbox{ when $d\ge 5$, }
  \end{equation}
  \begin{equation}\label{eq:dgc:dim X4}
    \#X(\Q,H) \leq c  H^{\dim X} (\log H)^\kappa  \mbox{ when $d=4$, }
  \end{equation}
  and
  \begin{equation}\label{eq:dgc:dim X3}
    \#X(\Q,H) \leq c H^{\dim X -1 + 2/\sqrt{3}} (\log H)^\kappa  \mbox{ when $d=3$.}
  \end{equation}
\end{Thm}

As before we use a variant for counting on an affine variety
$X\subset \A_\Q^n$, writing
$$
X(\Z,H) := \#\{x\in \Z^n\mid x\in X(\Q) \mbox{ and
} |x_i|\leq H \mbox{ for each $i$} \}.
$$

\begin{Thm}[Uniform dimension growth, affine case]\label{thm:0.4}
  Given $n>2$, there exist constants $c=c(n)$, $e=e(n)$ and an
  absolute constant $\kappa$, such that for all polynomials $f$ in
  $\Z[y_1,\ldots,y_n]$ satisfying that the homogeneous part of highest
  degree $h(f)$ of $f$ is absolutely irreducible and such that the
  degree $d$ of $h(f)$ is at least $3$, one has for all $H\geq 2$,
  with $X$ given by $f=0$,
  \begin{equation}\label{eq:dgc-aff:dim X5}
    \#X(\Z,H) \leq c d^e H^{n-2} \mbox{ when $d\ge 5$, }
  \end{equation}
  \begin{equation}
    \#X(\Z,H)
    \leq c  H^{n-2} (\log H)^\kappa  \mbox{ when $d=4$, }
  \end{equation}
  and
  \begin{equation}
    \#X(\Z,H)
    \leq c  H^{n-3+2/\sqrt{3}} (\log H)^\kappa  \mbox{ when $d=3$.}
  \end{equation}
\end{Thm}


The cases $d=4$ of Theorems \ref{thm:dcgdegree} and
\ref{thm:0.4} are new, compared to Theorems 1 and 4 of \cite{CCDN-dgc}
for $d\ge 5$, and, to Theorem 0.4 of \cite{Salberger-dgc} for degree $4$. Our short proof applies to all degrees at once, and
streamlines the proofs from \cite{Salberger-dgc},
\cite{Brow-Heath-Salb} and \cite{CCDN-dgc} by using our new Theorem
\ref{thm:main:gen}. Note that the cases with $d=3$ of Theorems \ref{thm:dcgdegree} and
\ref{thm:0.4} are not the sharpest known, see Theorem 16.3 and Corollary 16.4 of \cite{Salb:d3}, which replace $2/\sqrt{3}-1$ by $1/7+\e$; we include the cases $d=3$ nevertheless because they follow from our simplified and streamlined approach for all $d$. 
Note that the uniform degree $3$ cases of \cite{Salb:d3} are conjecturally still not
optimal: see the non-uniform result \cite[Theorem~0.1]{Salberger-dgc}
and the uniform Conjecture 2 of \cite{Heath-Brown-Ann} which
remains open in degree $3$.
Priorly, uniform dimension growth (with $\varepsilon$) was obtained in
\cite{Brow-Heath-Salb} for degree $d\ge 6$ and in
\cite{Salberger-dgc} for degree $d\ge 4$ and partially in degree $3$ (and in a sharpened form in \cite{Salb:d3}); the $\varepsilon$ was
removed for degrees $d\ge 5$ and the dependence on $d$ was made
polynomial in $d$ in \cite
{CCDN-dgc}.

We base ourselves on the following result which is obtained in \cite{CCDN-dgc} from improved forms of \cite[Theorem 1.3]{Walsh} and \cite[Theorem 1.2]{Salberger-dgc} using \cite[Remark 2.3]{Ellenb-Venkatesh}, the improvements being the explicit polynomial dependence on $d$ in the upper bound.

\begin{Prop}[\cite{CCDN-dgc}]\label{prop4.2.1}
  Fix an integer $n>2$. Then there exist $c=c(n)$ 
  such   that the following holds for all $f,H,d$.  Let $f$ in
  $\Z[x_1,\ldots,x_{n}]$ be irreducible, primitive and of degree
  $d\ge 1$.  Write $f_d$ for the degree $d$ homogeneous part of
  $f$. Suppose that $f_d$ is absolutely irreducible. Fix
  $H\geq 1$. Then there is a polynomial $g$ in $\Z[x_1,\ldots,x_{n}]$
  of degree at most
  \begin{equation}\label{eq:Walsh}
    c d^e H^{\frac{1}{d^{1/(n-1)}}},
  \end{equation}
  not divisible by $f$, and vanishing on all points $x$ in $\Z^{n}$
  satisfying $f(x)=0$ and $|x_i|\leq H$ for all $i=1,\ldots,n$, with furthermore $e=4-1/n$.
\end{Prop}
\begin{proof}
  This follows from Corollary 3.2.3 and Proposition 4.2.1 from
  \cite{CCDN-dgc}, by estimating the minimum appearing in Proposition
  4.2.1 from \cite{CCDN-dgc} by $d^2b(f)$, and by using
  $b(f)\leq b(f_d)$, with $b(f)$ and $b(f_d)$ as in Definition 3.2.1
  of \cite{CCDN-dgc}.
\end{proof}

\begin{proof}[Proof of Theorem \ref{thm:0.4}]
  Let us first prove the base case $n=3$. Given $f$ and $H$, create a
  polynomial $g$ by Proposition \ref{prop4.2.1}. Thus, $g$ vanishes on
  all points in $X(\Z,H)$,
  \begin{equation}\label{eq:deg:g}
    \deg g \le c d^e H^{\frac{1}{\sqrt{d}}},
  \end{equation}
  and, $g$ is not divisible by $f$.  Let $C$ be a (reduced)
  irreducible component of the 
  intersection of $f=0$ with $g=0$ and denote this intersection by $\cC$.
  If $C$ is of degree $\delta>1$, then
  \begin{equation}\label{eq:delta}
    \#C(\Z,H) \leq c'  \delta^{3} H^{1/\delta} (\log H +\delta)
  \end{equation}
  by Theorem 3 of \cite{CCDN-dgc}, for some universal constant $c'$.
  By Proposition 4.3.3 of \cite{CCDN-dgc}, the total contribution of
  integral curves $D$ in $\cC$ of degree $1$ is at most
  \begin{equation}\label{eq:degree1}
    c''d^{e''} H
  \end{equation}
  for some universal constants
  $c'',e''$. 
  Suppose that $C_1,...,C_k$ are (reduced) irreducible components of
  $\cC$, and that $\deg(C_i)>1$ for all $i$. Furthermore, fix $m$ such that
  that $\deg(C_i)\leq\log H$ for all $1\leq i\leq m$ and
  $\deg(C_i)>\log H$ for all $i>m$.
  For all $i$ with $1\leq i\leq m$ we have
  \begin{equation}\label{eq:delta1}
    \#C_i(\Z,H) \leq c' H^{1/2} (\log H +1),
  \end{equation}
  from (\ref{eq:delta}) and a basic reasoning as for
  \cite[Eq. (4-3-4)]{CCDN-dgc}.

  On the other hand, if $\delta> \log H$ then $H^{\frac{1}{\delta}}$
  is bounded, and one finds
  \begin{equation}\label{eq:delta2}
    \sum_{m+1\leq i\leq k}\#C_i(\Z,H) \leq c'''d^{e'''} H^{2/\sqrt{d}} (\log H)^\kappa
  \end{equation}
  for some universal constants $c''',e''',\kappa$, by Theorem \ref{thm:main:gen} and by
  using that $\delta \le c d^{e+1} H^{\frac{1}{\sqrt{d}}}$.
  Putting these estimates together with
  $m\le c d^{e+1} H^{\frac{1}{\sqrt{d}}}$ proves the case $n=3$ of the
  theorem, for all values of $d$.  For general $n>3$, one follows
  induction on the dimension by cutting with well-chosen hyper planes
  exactly as in the proof of \cite[Theorem 4]{CCDN-dgc}.
\end{proof}

\begin{proof}[Proof of Theorem \ref{thm:dcgdegree}]
  Suppose first that $X$ is a hypersurface in $\P^n_\Q$ given by
  $f=0$.  Write $X_a$ for the affine hypersurface in $\A^{n+1}_\Q$
  given by $f=0$. Now Theorem \ref{thm:dcgdegree} for $X$ follows from
  Theorem \ref{thm:0.4} for $X_a$ since one trivially has
$$
\#X(\Q,H)\leq \#X_a(\Z,H).
$$
The general case follows by the projection argument exactly as for
\cite[Theorem 1]{CCDN-dgc}.
\end{proof}




\bibliographystyle{plain} \bibliography{nrefs}

\end{document}